\documentclass[a4paper,12pt]{amsart}
\usepackage{epsfig}
\usepackage{amsmath, amsfonts }
\usepackage{amssymb}
\usepackage{amscd}
\usepackage{eufrak}
\usepackage{amssymb}
\usepackage{latexsym}
\usepackage[all]{xy}

\author{Ndouné Ndouné  }
\title{on involutive cluster automorphisms  }

\newtheorem{theo}{Theorem}[section]
\newtheorem{prop}{Proposition}[section]
\newtheorem{definition}{Definition}[section]
\newtheorem{expl}{Example}[section]

\newtheorem{lem}{Lemma}[section]
\newtheorem{cor}{Corollary}[section]
\newtheorem{rem}{Remark}[section]
\usepackage{graphicx}
\usepackage{t1enc}

\begin{document}
\maketitle

\begin{abstract}
We construct a special embedding of the translation quiver $\mathbb{Z}Q$ in the three-dimensional affine
space $\mathbb{R}^{3}$ where $Q$ is a finite connected acyclic quiver and $\mathbb{Z}Q$ contains a local slice whose quiver is isomorphic to
the opposite quiver $Q^{op}$ of $Q.$ Via this embedding, we show that there exists an involutive anti-automorphism  of the
translation quiver $\mathbb{Z}Q.$ As an immediate consequence, we characterize explicitly the group of cluster
automorphisms of the cluster algebras of seed $(X,Q)$, where $Q$ and $Q^{op}$ are mutation equivalent.
\end{abstract}

\maketitle

\section{Introduction}
Cluster algebras were invented by Fomin and Zelevinsky $[\ref{fz}]$ in order to understand the
canonical bases of quantized enveloping algebras associated to semi-simple Lie algebras [\ref{lusz}].
 A cluster algebra
is a commutative ring with a distinguished set of generators, called $cluster$ $variables$. The set of all cluster variables
 is constructed recursively from an initial set using a procedure
called $mutation$.

       Today the theory of cluster algebras relates with numerous fields of mathematics, for example,
 representation theory of algebras, Lie theory, Poisson geometry, Teichmüller theory and integrable systems.
Following $[\ref{ashifsh}, \ref{asdush}]$, we
 are interested in studying the involutive cluster automorphisms. The objective of this paper is to solve an open problem
 posed in [\ref{ashifsh}]; it was conjectured that if $Q$ is a finite connected acyclic quiver which is mutation equivalent to $Q^{op}$, then the
group of cluster
automorphisms of the cluster algebra of a seed $(X,Q)$ is a semi-direct product of the subgroup of the direct cluster
automorphisms with $\mathbb{Z}_{2}.$ A partial answer was given in [\ref{ashifsh}], in the case when the underlying graph of $Q$ is a tree.
We prove it here in general.

               In this paper, we construct an embedding of a repetitive translation quiver $\mathbb{Z}Q$ in $3$-space
 $\mathbb{R}^{3}$, when $Q$ is a finite connected acyclic quiver and the translation quiver contains a local slice whose quiver
 is isomorphic to $Q^{op}.$ In particular, we construct an embedding of the transjective component of the Auslander
 Reiten quiver of the cluster category $\mathcal{C}_{Q}$. We then use this construction to prove the existence of an
 involutive inverse cluster automorphism of a cluster algebra with a seed $(X,Q)$, where $Q$ is a finite connected acyclic quiver which
 is mutation equivalent to $Q^{op}.$

      Our paper is organized as follows. After an introductory section 2,
  we construct in section 3, the required embedding $\eta$ of $\mathbb{Z}Q$ in $\mathbb{R}^{3}$, such that
 if $\Sigma$
and $\widetilde{\Sigma}$ are two local slices whose quivers are respectively isomorphic to $Q$ and $Q^{op}$, then $\eta(\Sigma)$ and
$\eta(\widetilde{\Sigma})$
are symmetric with respect to a particular plane.

    In section 4, we prove that, if $Q$ is a finite acyclic quiver which is mutation equivalent to $Q^{op}$, then the
group of cluster
automorphisms of a cluster algebra of a seed $(X,Q)$ is a semi-direct product of the subgroup of the direct cluster
automorphisms with $\mathbb{Z}_{2}.$\\

\begin{center}ACKNOWLEDGMENTS
\end{center}
        This article is part of my PhD thesis, under the supervision of Ibrahim Assem and Vasilisa Shramchenko. 
I would like to thank them deeply for their patience and availability. 

\section{Cluster Automorphisms}

The notion of cluster automorphisms was introduced in
[\ref{ashifsh}] and it is connected with quiver automorphisms.
We recall that a $quiver$ is a quadruple $Q=(Q_0,Q_1,s,t)$ consisting of two sets,
$Q_0$(whose elements are called $vertices$) and $Q_1$(whose elements are called $arrows$),
and of two maps $s,t:Q_1\longrightarrow Q_0$ associating to each arrow $\alpha\in Q_1$
its $source$ $s(\alpha)$ and its $target$ $t(\alpha).$ If $i=s(\alpha)$ and $j=t(\alpha),$
we denote this situation by $ i\overset{\alpha}{\longrightarrow}
j$. Given a vertex $i$, we set  $i^{+} = \{\alpha\in Q_1| s(\alpha)=i\}$ and
$i^{-} = \{\alpha\in Q_1| t(\alpha)=i\}$.\\
Let $Q$ be a finite connected quiver without oriented cycles of length one or two.
Let $n=|Q_0|$ denote the number of vertices in $Q,$ $X =\{x_1,x_2, ..., x_{n}\}$ be a set
of $n$ variables and denote the vertices by  $Q_0 =\{1,2, ..., n\}$, where we agree that the point
$i$ corresponds to the variable $x_i.$ We consider the field $\mathcal{F}=\mathbb{Q}
(x_1,x_2, ..., x_{n})$ of rational functions in $x_1,x_2, ..., x_{n}$ (with rational coefficients).
The $cluster$ $algebra$ $\mathcal{A}=\mathcal{A}(X,Q)$ is the $\mathbb{Z}$-subalgebra of $\mathcal{F}$
defined by a set of generators obtained recursively from $X$ in the following manner. Let $k$ be such
that $1\leq k\leq n.$
We define the mutation $\mu_{x_{k},X}$ (or $\mu_{x_{k}}$ or $\mu_k$ for brevity if there is no
ambiguity) of the pair $(X,Q)$ to be a new pair $(X',Q')$, that is
$\mu_{k}(X,Q)=(X^{'},Q^{'})$, where $Q^{'}$ is the quiver obtained from $Q$
by performing the following operations:\\
$(a)$ for any path  $i\longrightarrow k\longrightarrow j$ of length two having $k$ as mid-vertex,
we insert a new arrow $i\longrightarrow j.$\\
$(b)$ all arrows incident to the vertex $k$ are reversed, \\
$(c)$ all cycles of length two are deleted.

    On the other hand, $X^{'}$ is a set of variables defined as follows.\\
$X^{'}=(X\setminus \{x_k\})\cup \{x_{k}^{'}\}$ where $x_{k}^{'}\in \mathcal{F}$ is obtained from $X$ by the
so-called $exchange$ $relation$: $$x_{k}x_{k}^{'}=\underset{\alpha\in i^{+}} \prod x_{t(\alpha)}+
\underset{\alpha\in i^{-}} \prod x_{s(\alpha)}.$$

     Let $\mathcal{X}$ be the union of all possible sets of variables obtained from
$X$ by successive mutations. Then the $cluster$ $algebra$ $\mathcal{A}=\mathcal{A}(X,Q)$ is the $\mathbb{Z}$-subalgebra of
$\mathcal{F}$
generated by $\mathcal{X}.$

 Each pair $(\tilde{X},\tilde{Q})$ obtained from $(X,Q)$ by successive mutations is
called a $seed$, and the set $\tilde{X}$ is called a $cluster$. The elements $\tilde{x}_{1},\tilde{x}_{2}, ... ,
\tilde{x}_{n}$ of a cluster $\tilde{X}$ are called $cluster$
$variables$. Each cluster is a transcendence basis for the field $\mathcal{F}$. The pair $(X,Q)$ is called
the $initial$ $seed$, and $X$ is called the $initial$ $cluster$.

    Gekhtman, Shapiro and Vainshtein showed in [\ref{geshva}] that for every seed $(\tilde{X},\tilde{Q})$
the quiver $\tilde{Q}$ is uniquely defined by the cluster $\tilde{X}$, and we use the notation $Q(\tilde{X})$
for the quiver of a given cluster $\tilde{X}$. We write $p_x$ for the point of $Q(\tilde{X})$ corresponding to
the cluster variable $x\in \tilde{X}.$
 Following [\ref{fz}] we recall that
the Laurent phenomenon asserts that each cluster variable in $\mathcal{A}$ can be expressed as a
Laurent polynomial in  $x_i$, with $1\leq i\leq n$, that is, every cluster variable is of the form

\begin{displaymath} \frac{p(x_{1},x_{2}, ... , x_{n})}{\prod\limits_{l=1}^{n} x_{l}^{d_l}},
 \end{displaymath} where $p\in \mathbb{Z}[x_{1},x_{2}, ... , x_{n}]$, and $d_l$ is a positive integer(for $1\leq l\leq n$).

As a consequence, $\mathcal{A}\subseteq \mathbb{Z}[x_1^{\pm 1},...,x_{n}^{\pm 1}]$, see [\ref{fz}] for more details.
The following definition is due to [\ref{ashifsh}].

\begin{definition}
Let $\mathcal{A}$ be a cluster algebra, and let $f:\mathcal{A}\longrightarrow \mathcal{A}$ be an automorphism of
$\mathbb{Z}$-algebras.  Then f is called cluster automorphism if there exists a seed $(X,Q)$ of $\mathcal{A}$
such that the following conditions are satisfied:\\
(CA1) $f(X)$ is a cluster;\\
(CA2) $f(X)$ is compatible with mutations, that is, for every $x\in X$, we have $f(\mu_{x,X}(x))=
\mu_{f(x),f(X)}(f(x)).$

\end{definition}

Now we state an equivalent characterization of cluster automorphisms see [\ref{ashifsh},Lemma 2.3].

\begin{lem}
Let $f$ be a $\mathbb{Z}$-algebra automorphism of $\mathcal{A}.$ Then f is a cluster automorphism if and only if
there exists a seed $(X,Q)$ such that $f(X)$ is a  cluster and one of the following two conditions is satisfied:
\\
(a) there exists an isomorphism of quivers $\varphi:Q\longrightarrow Q(f(X))$ such that $\varphi(p_{x})=p_{f(x)}$
for all $p_{x}\in Q_0$, or \\
(b) there exists an isomorphism of quivers $\varphi:Q^{op}\longrightarrow Q(f(X))$ such that $\varphi(p_{x})=p_{f(x)}$
for all $p_{x}\in Q_0$.
\end{lem}
A cluster automorphism $f$ is called $direct$ if it satisfies condition $(a)$ of the above lemma, and $f$ is called
$inverse$ if it satisfies condition $(b)$.

  The set of all cluster automorphisms forms a group under composition denoted
by Aut$(\mathcal{A})$ and called the group of $cluster$ $automorphisms$. The set of all direct automorphisms is denoted by
Aut$^{+}(\mathcal{A})$; it is a normal subgroup of Aut$(\mathcal{A})$ of index at most two.
For more results and properties of cluster automorphisms, we refer the reader to [\ref{ashifsh}].

 \section{Involutive automorphisms of quivers}
In this section, we assume that $Q$ is a finite connected acyclic quiver.
Let $K$ be an algebraically closed
field, $KQ$ the path algebra of $Q$ and mod$KQ$ the category of finitely generated right  $KQ$-modules. For the properties of
mod$KQ$ and its Auslander-Reiten quiver $\Gamma$(mod$KQ$) we refer the reader to [\ref{auresm}, \ref{asisk}].
We denote by $\mathcal{D}^{d}$(mod$KQ$) the bounded derived category over mod$KQ.$ This is a triangulated Krull-Schmidt category
having a Serre duality and hence almost split triangles. Since $KQ$ is hereditary, the Auslander Reiten quiver
$\Gamma(\mathcal{D}^{d}$(mod$KQ$)) of
$\mathcal{D}^{d}$(mod$KQ$) is well-understood, see [\ref{ha}].

     Following [\ref{asisk}] we give the definitions and some basic properties of a well-known translation quiver.
If $Q$ is a connected acyclic quiver, the $repetitive$ $quiver$ $\mathbb{Z}Q$ is the quiver whose set of vertices
is $(\mathbb{Z}Q)_{0}=\{(m,\nu)| m\in \mathbb{Z} , \nu\in Q_0\}$  and the set of arrows is $(\mathbb{Z}Q)_{1}=\{(m,\alpha):
(m,\nu)\rightarrow (m,\nu')| m\in \mathbb{Z} , \alpha:\nu\rightarrow \nu'\}\cup \{(m,\alpha'):
(m,\nu')\rightarrow (m-1,\nu)| m\in \mathbb{Z} , \alpha:\nu\rightarrow \nu'\}$. The map $\tau: \mathbb{Z}Q\longrightarrow \mathbb{Z}Q$ defined
 by $\tau(m,\nu)=(m-1,\nu)$ is called the
translation of the quiver $\mathbb{Z}Q$. The pair $(\mathbb{Z}Q,\tau)$ is
a $translation$ $quiver,$ for more details we refer to [\ref{asisk}]. The translation quiver $\mathbb{Z}Q$ is stable because its translation is
defined everywhere.
We recall that an automorphism of a translation quiver is a quiver automorphism commuting with the translation $\tau$.

    The quiver $\Gamma(\mathcal{D}^{d}$(mod$KQ$)) consists of two types of components: the regular
components and the transjective components, the latter are of the form $\mathbb{Z}Q$, see[\ref{ha}].
The $cluster$ $category$ $\mathcal{C}_{Q}$ of
Buan-Marsh-Reineke-Reiten-Todorov
 is defined to be the orbit category of
$\mathcal{D}^{d}$(mod$KQ$) under the action of the automorphism $\tau^{-1}[1]$, where $\tau$ is the Auslander-Reiten translation
and $[1]$ is
the shift of $\mathcal{D}^{d}$(mod$KQ$), see [\ref{bmrrt}]. Then $\mathcal{C}_{Q}$ is also a triangulated Krull-Schmidt category
having almost split triangles. Moreover, $\mathcal{C}_{Q}$ is a 2-Calabi-Yau category [\ref{bmrrt}].
 The Auslander-Reiten quiver $\Gamma(\mathcal{C}_{Q})$ of $\mathcal{C}_{Q}$ is the quotient of
$\Gamma(\mathcal{D}^{d}$(mod$KQ$)) under the action of the quiver automorphism induced by $\tau^{-1}[1]$.
Then, the combinatorics of cluster
variables is encoded in the cluster category.
That is, if $\mathcal{A}=\mathcal{A}(X,Q)$ is a cluster algebra, with $Q$ an acyclic
quiver having $n=|Q_0|$ points, there exists a map $X_{?}:\mathcal{C}_{Q}\longrightarrow \mathbb{Z}[x_1^{\pm 1},...,x_{n}^{\pm 1}]$
called the canonical $cluster$ $character$, or the $Caldero-Chapoton$ $map$. The map $X_{?}$ induces a bijection between the indecomposables
 $M$(up to isomorphism) in $\mathcal{C}_{Q}$ which have no self-extensions and the cluster variables $X_{M}$, see [\ref{cc},\ref{ck}].
 Under this
 bijection, clusters correspond to cluster-tilting objects  in $\mathcal{C}_{Q}$.

   It has been shown in [\ref{ashifsh}] that, if $\mathcal{A}$ is a cluster algebra with a seed $(X,Q)$ such that the underlying
graph of $Q$ is a tree, then there exists an inverse cluster automorphism $\sigma \in$ Aut$(\mathcal{A})$ of order two and therefore
Aut$(\mathcal{A})$=Aut$^{+}(\mathcal{A})\rtimes \mathbb{Z}_2$. In this section we show that this statement is valid for
any acyclic quiver which is mutation equivalent to its
opposite quiver.

      Let $j$ be an integer, a subquiver of $\mathbb{Z}Q$ of the form $(j,Q)$ is a full subquiver of $\mathbb{Z}Q$ whose set of vertices is
$(j,Q)_{0}=\{(j,\nu)| \nu\in Q_0\}$.

         We denote by $C^{j}$ the subquiver of the quiver $\mathbb{Z}Q$ of the form $(j,Q)$.

         The notion of $local$ $slice$ was introduced in [\ref{asbrsh}] for the connected components of Auslander-Reiten quivers of the
categories of modules. For the repetitive quiver $\mathbb{Z}Q$, the notion of local slice coincides with the notion of $section$ as defined
in [\ref{asisk}, p.302].

    We consider the $3$-dimensional Euclidean space $\mathbb{R}^{3}$  with coordinates $(x,y,z)$, where the frame of reference has a
negative orientation as shown in Fig 1.\\

Our main result will be obtained as a corollary of the following theorem.

\begin{theo}
Let $Q$ be a finite connected acyclic quiver such that the translation quiver $\mathbb{Z}Q$ contains two local slices $\Sigma$
and $\widetilde{\Sigma}$
whose quivers are respectively isomorphic  to $Q$ and $Q^{op}.$ Then there is an embedding $\eta:\mathbb{Z}Q\longrightarrow
\mathbb{R}^{3}$ and an oblique reflection with respect to the plane $(xOz)$ which maps $\eta(\Sigma)$ and $\eta(\widetilde{\Sigma})$
 to each other.

\end{theo}

\begin{proof}

The translation quiver $\mathbb{Z}Q$ contains two local slices $\Sigma$ and $\widetilde{\Sigma}$
whose quivers are isomorphic respectively to $Q$ and $Q^{op}.$
Since the translation quiver $\mathbb{Z}Q$ is stable, we can assume without loss of generality that $\Sigma$ and $\widetilde{\Sigma}$
are such that  $\Sigma\cap
\widetilde{\Sigma}\neq \emptyset$ and $\Sigma$ is chosen to be the copy of $Q$ at zero, that is $\Sigma=(0,Q)=C^{0}.$

Because $\widetilde{\Sigma}$ is a subquiver of $\mathbb{Z}Q$ and $\Sigma_{0}\cap \widetilde{\Sigma}_{0}\subset C^{0},$
there exists a maximal nonnegative integer $k$, such that $C^{-j}\cap
\widetilde{\Sigma}\neq \emptyset$ for all $j$ with $0\leq j\leq k$.
We use the following notation for the vertices of $\Sigma$: the vertex $(0,\nu_i)$ is denoted by $(0,\nu_{ji})$, where $j$ is such that
the vertex $(-j,\nu_i)$ belongs to $\widetilde{\Sigma},$ for
$0\leq j\leq k.$  We extend this notation to the translation quiver $\mathbb{Z}Q$ that is $(-m,\nu_{ji})=\tau^{m}(0,\nu_{ji}).$
We denote by $s_{j}$, where $0\leq j\leq k,$ the number of vertices $(0,\nu_{ji})$. Thus we have $\sum\limits_{j=0}^{k} s_{j}=|Q_0|.$

     The proof is completed in the following three steps $(a)$ $(b)$ $(c)$.

$(a)$ If $A_1, A_2,..., A_h$ are points of $\mathbb{R}^{3}$, we denote by $[A_1A_2...A_h]$ the convex hull of the points
$A_1, A_2,..., A_h.$

   Let $n=|Q_0|$ and define $\mathbb{P}_{nk}$ to be the convex hull of
the points $O,D_{0}, D_{k},E_{k},E_{0}^{n},E_{k}^{n},D_{0}^{n}, D_{k}^{n}$ whose coordinates are given by:

$O=(0,0,0)$, $D_{0}=(-1,0,0),$ $D_{k}=(-1,k,0),$ $E_{k}=(0,k,0),$
$E_{0}^{n}=(0,0,n),$ $E_{k}^{n}=(0,k,n),$ $D_{0}^{n}=(-1,0,n),$ $D_{k}^{n}=(-1,k,n)$; then $\mathbb{P}_{nk}$ is a rectangular
parallelepiped (see Fig.1).

            \begin{figure}
\includegraphics{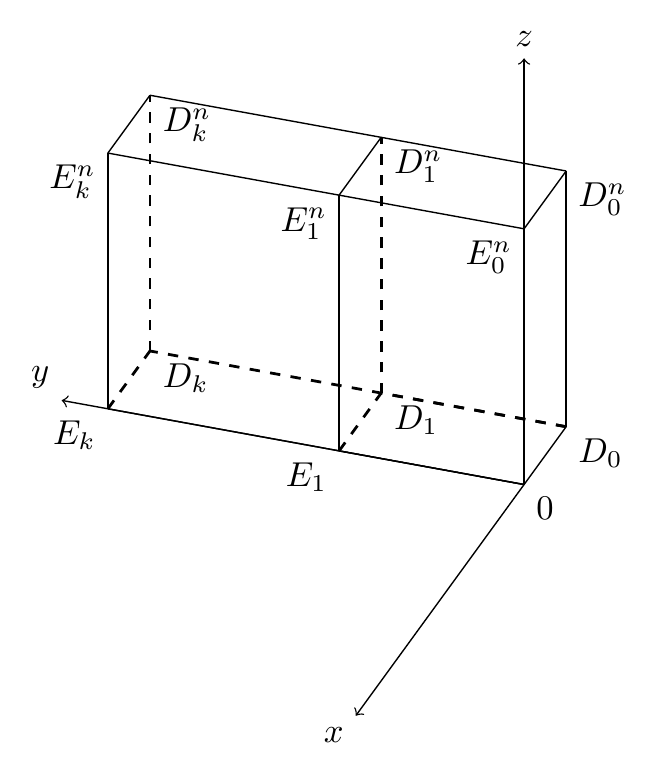}
\center{Fig.1} \end{figure}

             We want to embed $\Sigma$ in $\mathbb{R}^{3}$ in such a way that the vertices of the embedded quiver $\eta(\Sigma)$
belong to $\mathbb{P}_{nk}\setminus [D_{0}D_{k}D_{k}^{n}D_{0}^{n}]$ that is, to the
parallelepiped $\mathbb{P}_{nk}$ from which we remove the face $D_{0}D_{k}D_{k}^{n}D_{0}^{n}$. We set $H=(0,0,n-1)$ and $T=(-1,0,n-1).$
Let $(\mathcal{P}_{m})_{m\in \mathbb{Z}}$ be the family of parallel
planes in $\mathbb{R}^{3}$ defined by the equations $\mathcal{P}_{m}:y=m$.
We decompose $\Sigma$ into $k+1$ non-intersecting subquivers $\Sigma^{j}$ and construct the embedding of each subquiver separately.

         The construction of  $\eta(\Sigma)$ is thus done in three steps $(a_i),$ $(a_{ii})$ and $(a_{iii})$ as follows.

    $(a_i)$ Construction of
$\eta(\Sigma^{0})$, where $\Sigma^{0}=\Sigma\cap \widetilde{\Sigma}$. Then
 $\Sigma^{0}$ is the full subquiver of $\Sigma$ whose set of vertices is
$\Sigma^{0}_{0}=\{(0,\nu_{0l})| 1\leq l\leq s_0\}$.
We map the vertices $(0,\nu_{01})$, $(0,\nu_{02})$,..., $(0,\nu_{0s_0})$ of $\Sigma$,
to distinct  points $A_{01}$, $A_{02}$,..., $A_{0s_0}$ in $[E_{0}^{n}D_{0}^{n}TH]\setminus [TD_{0}^{n}]$ such that for all $i,j\in \{1,2,...,s_0\}$ the line
$(A_{0i}A_{0j})$ is not parallel to a line $(OD_0)$.
Each arrow $\alpha$ of $\Sigma$ from the vertex $(0,\nu_{0i})$ to the
vertex $(0,\nu_{0j})$ is mapped to
an arrow $\bar{\alpha}$ from $A_{0i}$ to $A_{0j}$.
In other words, we set
$\eta(0,\nu_{0i})=A_{0i}$ with $1\leq i\leq s_0$ and $\eta(\alpha)=\bar{\alpha}.$ The quiver $\eta(\Sigma^{0})$ obtained in the plane
$\mathcal{P}_{0}$ is isomorphic to $\Sigma^{0},$ we call it an embedding of $\Sigma^{0}$.

    $(a_{ii})$ Now we construct $\eta(\Sigma^{1})$, where $\Sigma^{1}$ is the full subquiver of $\Sigma$ whose set of vertices is
$\Sigma^{1}_{0}=\{(0,\nu_{1l})| 1\leq l\leq s_1\}$.

         We denote by $t_{\vec{v}}$ the translation in $\mathbb{R}^{3}$ by the vector $\vec{v}=(0,1,-1).$

          Let $\Sigma^{11}_{0}$ be the set of all vertices $(0,\nu_{1l})$ of $\Sigma^{1}_{0}$ such that there exists at least one arrow from
a vertex $(0,\nu_{0i})$ in $\Sigma^{0}$ to
the vertex $(-1,\nu_{1l})$. And denote by $\Sigma^{12}_{0}$ the complement of
$\Sigma^{11}_{0}$ in $\Sigma^{1}_{0};$ then $\Sigma^{12}_{0}$ is the set of the elements $(0,\nu_{1l})$ in $\Sigma^{1}_{0}$ such that there is no arrow from
$\Sigma^{0}$ to $(-1,\nu_{1l})$.  Since the quiver $\Sigma$
is connected, we have $\Sigma^{11}_{0}\neq \emptyset$.

  For the construction of the embedding of the set $\Sigma^{11}_{0}$, we divide $\Sigma^{11}_{0}$ into two families of subsets as follows.

           The first family consists of the non-intersecting subsets $L_1$, $L_2,$...,$L_h$ of $\Sigma^{11}_{0}$ such that $|L_j|\geq 2$ and for
some $(0,\nu_{0i_{j}})\in \Sigma^{0}_{0}$ there exists an arrow to $(0,\nu_{0i_{j}})$ from each element $L_j$. The second family is given by the
remaining elements of $\Sigma^{11}_{0}$, the elements which do not belong to any set $L_j$ with $1\leq j\leq h$.

           Now we construct first the embedding of the elements of the set $\Sigma^{11}_{0}\setminus (L_{1}\cup L_{2}...\cup L_{h})$. For an element $(0,\nu_{1r})$
of $\Sigma^{11}_{0}\setminus (L_{1}\cup L_{2}...\cup L_{h})$, we take any arrow originating in $(0,\nu_{1r})$ whose target is
$(0,\nu_{0i_{r}})\in \Sigma^{0}_{0}.$  Then the vertex $(0,\nu_{1r})$ is mapped to the point $A_{1r}:=t_{\vec{v}}(A_{0i_{r}}).$

             For the construction of the embedding of the elements of the set $L_j$ with $1\leq j\leq h$ , we choose an element $(0,\nu_{1l_{j}})$ in $L_j$ and we
map it to a point $A_{1l_{j}}:=t_{\vec{v}}(A_{0i_{j}})$ with $A_{0i_{j}}=\eta(0,\nu_{0i_{j}})$. In order to define the embedding for the remaining elements of
$L_{j}$, we proceed as follows.
Let $\Delta_j$ be
the line parallel to $(Ox)$ passing through $A_{1l_{j}}$ and let $[FG]$ be the intersection of $\Delta_j$ and the rectangle $[E_{1}D_{1}D_{1}^{n}E_{1}^{n}]$
(we recall that $A_{1l_{j}}\in [E_{1}D_{1}D_{1}^{n}E_{1}^{n}]$).
The elements of $L_j\setminus \{(0,\nu_{1l_{j}})\}$ are mapped to arbitrary distinct points in $[GF]\setminus \{G,F,A_{1l_{j}}\}.$
The elements of $L_i$, $i\neq j$, $1\leq i\leq h$ are mapped to the distinct point of the line $\Delta_i$ constructed as above. We have $\Delta_i\cap \Delta_j=\emptyset$, for $i\neq j.$

          It remains to construct the elements of the set $\Sigma^{12}_{0}$.
We map the elements of $\Sigma^{12}_{0}$ to distinct points
(distinct from the points associated with the elements of $\Sigma^{11}_{0}$) in
$t_{\vec{v}}([E_{0}^{n}D_{0}^{n}TH]\setminus [TD_{0}^{n}])$ such that any line containing at least two of them is not parallel to $(OD_0).$

     Each arrow $\alpha$ of $\Sigma$ from the vertex $(0,\nu_{1r})\in \Sigma^{12}_{0}$ to the
vertex $(0,\nu_{1l})\in \Sigma^{1}_{0}$ is mapped to
an arrow $\bar{\alpha}$ from $A_{1r}=\eta(0,\nu_{1r})$ to $A_{1l}=\eta(0,\nu_{1l})$ and each arrow $\beta$ of $\Sigma$ from the vertex $(0,\nu_{1j})\in \Sigma^{11}_{0}$ to the
vertex $(0,\nu_{0i})\in \Sigma^{0}_{0}$ is mapped to
an arrow $\bar{\beta}$ from $A_{1j}=\eta(0,\nu_{1j})$ to $A_{0i}=\eta(0,\nu_{1i}).$

   The obtained quiver $\eta(\Sigma^{1})$ is an embedding of the subquiver
$\Sigma^{1}$ into the plane
$\mathcal{P}_{1}:y=1.$ More precisely, $\eta(\Sigma^{1})$ is constructed in the rectangle $\mathcal{P}_{1}\cap \mathbb{P}_{nk}$.

      Thus, we have constructed the quiver $\eta(\overline{\Sigma^{0}\cup \Sigma^{1}})$, where $\overline{\Sigma^{0}\cup \Sigma^{1}}$ is the full subquiver of $\Sigma$ whose set of
vertices is $\Sigma^{0}_{0}\cup \Sigma^{1}_{0}.$

    $(a_{iii})$ Construction of $\eta(\Sigma^{j})$, where $\Sigma^{j}$, $2\leq j\leq k$ is the full subquiver of $\Sigma$ whose set of vertices is
$\Sigma^{j}_{0}=\{(0,\nu_{jl})| 1\leq l\leq s_j\}$.

     We construct $\eta(\Sigma^{j})$ assuming that $\eta(\Sigma^{j-1})$ is constructed.

     We reproduce the construction of step $a_{(ii)}$ with $\Sigma^{j}$ playing the role of $\Sigma^{1}$ and $\Sigma^{j-1}$ playing the role of $\Sigma^{0}.$
The vertices of $\Sigma^{j}$ are mapped by $\eta$ to the rectangle $[E_{j}D_{j}D_{j}^{n}E_{j}^{n}]\setminus [D_{j}D_{j}^{n}]$ and the arrows between two vertices $\nu$ and $\zeta$ of
$\Sigma_{0}$ are mapped to arrows between the corresponding images $\eta(\nu)$ and $\eta(\zeta).$  Therefore, we construct the embedding $\eta(\Sigma^{j})$ of $\Sigma^{j}$ into
the plane $\mathcal{P}_{j}:y=j$ for $0\leq j\leq k,$  more precisely $\eta(\Sigma^{j})\in \mathcal{P}_{j}\cap \mathbb{P}_{nk}.$ Thus, we have constructed the quiver
$\eta(\overline{\Sigma^{0}\cup \Sigma^{1}...\cup \Sigma^{j}})$, where
$\overline{\Sigma^{0}\cup \Sigma^{1}...\cup \Sigma^{j}}$ is the full subquiver of $\Sigma$ whose set of vertices is $\Sigma^{0}_{0}\cup \Sigma^{1}_{0}...\cup \Sigma^{j}_{0}.$
In other words, we have constructed the embedding of $\overline{\Sigma^{0}\cup \Sigma^{1}...\cup \Sigma^{j}}$ in
$[OD_{0}D_{0}^{n}E_{0}^{n}E_{j}^{n}E_{j}D_{j}D_{j}^{n}]\setminus [D_{0}D_{0}^{n}D_{j}^{n}D_{j}].$

    Since $\Sigma$ is finite, the process ends after $k$ iterations and we find the required embedded quiver $\eta(\Sigma)$. By construction we have
$\eta(\Sigma)=\eta(\overline{\Sigma^{0}\cup \Sigma^{1}...\cup \Sigma^{k}})$ and $|\eta(\Sigma)_{0}|=|\Sigma_{0}|.$
Also, we have an isomorphism of quivers $\eta(\Sigma)\cong \Sigma$ and  $\eta(\Sigma)$ is an embedding of $\Sigma$ in
$[OD_{0}D_{0}^{n}E_{0}^{n}E_{1}^{n}E_{1}D_{1}D_{1}^{n}]\setminus [D_{0}D_{0}^{n}D_{1}^{n}D_{1}]$.

   $(b)$ Construction of the quiver $\eta(\mathbb{Z}Q)$ in $\mathbb{R}^{3}.$
To define $\eta(\mathbb{Z}Q)$ in $\mathbb{R}^{3}$, we need first to construct each $\eta(C^{m})$, with $m\in \mathbb{Z}$, where $C^{m}$ is a copy of the quiver $Q.$

         Let $\vec{u}=(-1,-2,0)$ and $t_{\vec{u}}$ be the translation in $\mathbb{R}^{3}$ by vector $\vec{u}$.

      We extend $\eta$ to each copy $C^{m}$, with $m\in \mathbb{Z},$ by $\eta(C^{m})=t_{m\vec{u}}(\eta(\Sigma)).$ We complete the construction of $\eta(\mathbb{Z}Q)$
as one would do for the repetitive quiver $\mathbb{Z}Q$ using that $\eta(\tau)=t_{\vec{u}}$ . By the construction, there are no arrows which cross each other when
considered as segments in $\mathbb{R}^{3}$.
Thus we have extended the embedding $\eta$ to $\mathbb{Z}Q.$

       $(c)$ Finally we show that $\eta(\widetilde{\Sigma})$ is the image of $\eta(\Sigma)$ under the oblique reflection $S$ with respect to the plane $\mathcal{P}_{0},$
and the direction given by $\vec{u}$.

    If $A=(a,b,c)$ is a point of $\mathbb{R}^{3}$, then the image of $A$ under the oblique reflection $S$ is the point $S(A)=(a-b,-b,c).$

       Let $A_{jr}=(a_{jr},b_{jr},c_{jr})$ be a point of $\eta(\Sigma)$ such that $\eta(0,\nu_{jr})=A_{jr}$, with $(0,\nu_{jr})\in \Sigma_{0}$.
       Since $A_{jr}\in \mathcal{P}_{j}$, then the second coordinate of $A_{jr}$ is equal to $j$, and $S(A_{jr})=(a_{jr}-j,-j,c_{jr}).$

               By the definition of $\eta,$ we have $\eta(\tau)=t_{\vec{u}}$ and $\eta(\tau^{-1})=t_{-\vec{u}}.$

          Now we want to show that $\eta(\widetilde{\Sigma})=S(\eta(\Sigma)).$

         Let $N_{r}\in \widetilde{\Sigma}_0$, then $N_{r}=(-j,\nu_{jr})=\tau^{j}(0,\nu_{jr})$, with $1\leq r\leq s_{j}$, where
 $\sum\limits_{j=0}^{k} s_{j}=n.$ We have

$$\begin{array}{lllll}\eta(N_{r})&=&\eta(\tau^{j}(0,\nu_{jr}))\\ &=&t_{{j}\vec{u}}(A_{jr})\\
&=& (a_{jr}-j,-j,c_{jr})\\
&=&S(A_{jr}). \end{array}$$
Therefore $\eta(\widetilde{\Sigma}_0)\subseteq S(\eta(\Sigma_{0})).$
Since $|\eta(\widetilde{\Sigma}_0)|=|S(\eta(\Sigma_0))|$, then $\eta(\widetilde{\Sigma}_0)=S(\eta(\Sigma_0)).$
Because the quivers  $\Sigma$ and $\widetilde{\Sigma}$ are full subquivers of $\mathbb{Z}Q$, it follows that $S(\eta(\Sigma))$ and $\eta(\widetilde{\Sigma})$ are also
full subquivers of $\eta(\mathbb{Z}Q).$ Thus $S(\eta(\Sigma))=\eta(\widetilde{\Sigma})$.
This completes the proof.

\end{proof}

\begin{definition}
The quiver $\eta(\mathbb{Z}Q)$ is called a spatial representation of $\mathbb{Z}Q$ in $\mathbb{R}^{3}$.
\end{definition}

\begin{rem}

The reflection $S$ defined in $\mathbb{R}^{3}$ in step $(c)$ of the proof transforming $\eta(\Sigma)$ to $\eta(\Sigma^{op})$ does not induce
necessarily a quiver morphism, as shown in the following example.

\end{rem}

\begin{expl}
Let $Q$ be a quiver defined by:$$ \xymatrix{1\ar[rr]^{\alpha_{13}}\ar[dr]_{\alpha_{12}}&&3 \\
&2\ar[ur]_{\alpha_{23}}&}.$$
The quiver $\mathbb{Z}Q$ contains a local slice $\widetilde{\Sigma}$ isomorphic to $Q^{op}$ and that intersects the local slice $\Sigma=(0,Q)$.
The local slice $\Sigma$ is given by $$\xymatrix{(0,1)\ar[rr]^{(0,\alpha_{13})}\ar[dr]_{(0,\alpha_{12})}&&(0,3) \\
&(0,2)\ar[ur]_{(0,\alpha_{23})}&}$$ and the local slice $\widetilde{\Sigma}$ is given by $$\xymatrix{(0,3)\ar[rr]^{(0,\alpha'_{23})}\ar[dr]_{(0,\alpha'_{13})}&&(-1,2) \\
&(-1,1)\ar[ur]_{(-1,\alpha_{12})}&}.$$ We set $\eta(0,3)=(0,0,2)$, $\eta(0,2)=(0,1,1)$ and $\eta(0,1)=(0,1,0)$. We have $S(\eta(0,3))=(0,0,2)$, $S(\eta(0,2))=\eta(-1,2)$ and $S(\eta(0,1))=\eta(-1,1)$.
 The reflection $S$ maps the arrow $\eta(0,\alpha_{12}):\eta(0,1)\rightarrow \eta(0,2)$ to the arrow $\eta(-1,\alpha_{12}):\eta(-1,1)\rightarrow \eta(-1,2),$ with
$S(\eta(0,1))=\eta(-1,1)$ and $S(\eta(0,2))=\eta(-1,2).$ Thus $S$ does not induce an anti-morphism.\\
The reflection $S$ does not induce a quiver morphism, because for the arrow $\eta(0,\alpha_{23}):\eta(0,2)\rightarrow \eta(0,3)$ of $\eta(\Sigma)$, we have $S(\eta(0,\alpha_{23}))=\eta(0,\alpha'_{23}):\eta(0,3)\rightarrow \eta(-1,2).$
\end{expl}

\section{Involutive cluster automorphisms}

Before giving a first consequence of our result, we need to recall the notion of quiver with length appearing in [\ref{boga}, \ref{chpltr}].

Let $Q$ be a quiver, and let $\omega=\alpha^{\varepsilon_1}_{1}\alpha^{\varepsilon_2}_{2}...\alpha^{\varepsilon_s}_{s}$ be a walk, where
$\alpha_i\in Q_1$ and $\varepsilon_i\in \{-1,1\}$ for all $i$. We call $length$ of $\omega$ the integer defined by
$L(\omega)=\sum \limits_{i=1}^{s} L(\alpha^{\varepsilon_i}_{i})$, where $L(\alpha_i)=1$ and $L(\alpha^{-1}_i)=-1$ for all $i$.
Two paths in $Q$ having same source and target are called $parallel$ $paths$.
We say that $Q$ is a $quiver$ $with$ $length$ when any two parallel paths in $Q$ have the same length. Let $\nu, \xi$ be two vertices of the quiver
$Q$ with length and $n$ an integer.
The length $L(\nu,\xi)$ between $\nu$ and $\xi$ is $n$ if there exists a path from $\nu$ to $\xi$ of length $n.$

\begin{prop}

Let $Q$ be a finite connected acyclic quiver such that the translation quiver $\mathbb{Z}Q$ contains two local slices $\Sigma$
and $\widetilde{\Sigma}$
whose quivers are respectively isomorphic  to $Q$ and $Q^{op}.$
If the quiver $Q$ is a quiver with length, then the reflection $S$ induces an anti-automorphism of $\mathbb{Z}Q.$

\end{prop}

\begin{proof}
 We show that the reflection $S:\eta(\Sigma)\longrightarrow \eta(\widetilde{\Sigma})$ constructed in Theorem 3.1 is an anti-isomorphism
of quivers.
 Consider the embedded quiver $\eta(\mathbb{Z}Q)$ defined in Theorem 3.1.
  Note that all sinks of $\Sigma$ belong to the intersection $\Sigma\cap \widetilde{\Sigma}$ and therefore are mapped to the plane $\mathcal{P}_0.$
 As before $\Sigma$ is decomposed into $k+1$ non-intersecting subquivers $\Sigma^{j}$.

   Because the
point $A_{ji}$ belongs to the plane $\mathcal{P}_j: y=j$, then $j=\rho(A_{ji},\mathcal{P}_{0})$ where $\rho(A_{ji},\mathcal{P}_{0})$ is the Euclidean
distance from the point $A_{jl}$ to the plane
 $\mathcal{P}_{0}.$ Note that by the definition of the reflection $S$, we have also $j=\rho(S(A_{ji}),\mathcal{P}_{0})$.
Because the quiver $Q$ is a quiver with length, each subquiver $\Sigma^{l}$ of $\Sigma$ is discrete, that is, it contains no arrows. The only arrows of
$\Sigma$ are from a subquiver $\Sigma^{j}$ to
a subquiver $\Sigma^{j-1}$. Thus in spatial representation of $\eta(\mathbb{Z}Q)$ all arrows are between different planes and therefore are reversed by
the reflection $S$, which shows that $S$ is an anti-morphism. To make this reasoning formal, let $\bar{\alpha}_{il}: A_{ji}\longrightarrow A_{j-1l}$ be
 an arrow of $\eta(\Sigma).$
Then the arrow $(0,\alpha_{il}): (0,\nu_{ji})\longrightarrow (0,\nu_{j-1l})$ is an arrow of $\Sigma$ and we have $\eta(0,\alpha_{il})=\bar{\alpha}_{il}.$
On the other hand,
$(-j+1,\alpha'_{il}): (-j+1,\nu_{j-1l})\longrightarrow (-j,\nu_{ji})$ is an arrow of $\widetilde{\Sigma}$. Since $
\eta(-j+1,\nu_{j-1l})=S(A_{j-1l})$ and
$\eta(-j,\nu_{ji})=S(A_{ji})$, then $\eta(-j+1,\alpha'_{il})=\bar{\alpha}'_{il}$ is an arrow from $S(A_{j-1l})$ to $S(A_{ji})$. Thus
$S(\bar{\alpha}_{il})=\bar{\alpha}'_{il}$.
Hence $S$ is an anti-isomorphism from the quiver $\Sigma$ to the quiver $\widetilde{\Sigma}.$ The symmetry $S$ is extended to an
anti-automorphism of the quiver $\mathbb{Z}Q$.
\end{proof}

\begin{expl}
Let $Q$ be the Dynkin quiver of type $\mathbb{A}_{3}$ given by $ 1\overset{\alpha_{_{12}}}\longrightarrow  2\overset{\alpha_{_{23}}}\longrightarrow 3$.
The quiver $\mathbb{Z}Q$ contains a local slice $\widetilde{\Sigma}$ isomorphic to $Q^{op}$ which intersects the local slice $\Sigma=(0,Q)$.
The quiver $\Sigma$ is given by $$ (0,1)\overset{(0,\alpha_{_{12}})}\longrightarrow  (0,2)\overset{(0,\alpha_{_{23}})}\longrightarrow (0,3)$$ and
the quiver $\widetilde{\Sigma}$ is given by $$ (0,3)\overset{(0,\alpha'_{_{23}})}\longrightarrow  (-1,2)\overset{(-1,\alpha'_{_{12}})}\longrightarrow (-2,1).$$
 We set $\eta(0,3)=(0,0,3)$, $\eta(0,2)=(0,1,1)$ and $\eta(0,1)=(0,2,0)$. We have $S(\eta(0,3))=(0,0,3)$, $S(\eta(0,2))=\eta(-1,2)=(-1,-1,1)$ and $S(\eta(0,1))=\eta(-1,1)=(-2,-2,0)$.\\
 We define the anti-morphism $\sigma$ on the points by $\sigma(0,0,3)=(0,0,3)$, $\sigma(0,1,1)=(-1,-1,1)$ and $\sigma(0,2,0)=(-2,-2,0)$; and define on the arrows by $\sigma(\eta(0,\alpha_{23}))=\eta(0,\alpha'_{23})$
and $\sigma(\eta(0,\alpha_{12}))=\eta(-1,\alpha'_{12})$. The map $\sigma$ is an anti-isomorphism of quivers, which is extended to an anti-automorphism of quiver $\eta(\mathbb{Z}Q).$ The anti-automorphism $\sigma$ is a reflection.
Therefore, the reflection $\sigma$ induces an involutive anti-automorphism of the quiver $\mathbb{Z}Q.$
\end{expl}

             We can now establish the existence of a particular automorphism of the translation quiver $\mathbb{Z}Q.$
\begin{lem}
Let $Q$ be a finite connected acyclic quiver such that the translation quiver $\mathbb{Z}Q$ contains a local slice whose quiver is
 isomorphic to $Q^{op}.$ Then there exists an anti-automorphism $\sigma$ of $\mathbb{Z}Q$ such that $\sigma^{2}=1.$
\end{lem}

\begin{proof}
Let $Q$ be a finite connected acyclic quiver.

      If $Q$ is a quiver with length, the reflection $S$ induces an involutive anti-automorphism due to Proposition 4.1.

      Suppose that $Q$ is not a quiver with length. According to [\ref{chpltr}] the quiver $Q$ contains a subquiver of type $\tilde{A}_{pq}$, with $p\neq q.$ Let $\Sigma$ and $\widetilde{\Sigma}$ be two local slices of $\mathbb{Z}Q$ whose quivers are respectively isomorphic to $Q$ and $Q^{op}.$ Without
loss of generality, we can choose $\Sigma$
 and $\widetilde{\Sigma}$ such that $\Sigma_0\cap
\widetilde{\Sigma}_0\neq \emptyset$. Consider the spatial representation $\eta(\mathbb{Z}Q)$ of
$\mathbb{Z}Q$ constructed in the proof of Theorem 3.1. In the case without length, there exist planes $\mathcal{P}_j$ containing non-discrete quivers. It follows from Theorem 3.1 that
$\eta(\Sigma)$ and $\eta(\widetilde{\Sigma})$ are symmetric to each other with
respect to
the plane $\mathcal{P}_0$ of $\mathbb{R}^{3}$ via the reflection $S$.
    Let $\widetilde{Q}$ be the quiver defined by $\widetilde{\Sigma}=(0,\widetilde{Q})$ and denote by  $\tilde{C}^{m}$ the subquiver of $\mathbb{Z}\widetilde{Q}$ of the form $(m,\widetilde{Q}).$
Since $\widetilde{\Sigma}$ is a local slice of $\mathbb{Z}Q$, then $\tilde{C}^{m}$ is a local slice of $\mathbb{Z}Q$ for all integer $m.$

      By the choice of $\Sigma$ and $\widetilde{\Sigma}$, each sink of $\eta(\Sigma)$ is a source of $\eta(\widetilde{\Sigma}).$ Because  $\eta(\Sigma)$ is isomorphic to $Q$ and $\eta(\widetilde{\Sigma})$
is isomorphic to $Q^{op}$, then there exists an anti-isomorphism $\sigma:\eta(\Sigma)\longrightarrow \eta(\widetilde{\Sigma})$ such that every point of $\eta(\Sigma)\cap \eta(\widetilde{\Sigma})$ is invariant under
$\sigma.$ The anti-isomorphism $\sigma$ induces a permutation $\sigma_{j}$ in $S_{s_{j}},$ where $s_{j}$ is defined as in the proof of the Theorem 3.1. The permutation $\sigma_{j}$ is viewed as the restriction of
$S\sigma$ in $\mathcal{P}_{j}\cap \eta(\Sigma_0).$

       We define the map $\overline{\sigma}:\eta(\Sigma)\cup \eta(\widetilde{\Sigma})\longrightarrow \eta(\Sigma)\cup \eta(\widetilde{\Sigma})$ such that the restriction of $\overline{\sigma}$ on $\eta(\Sigma)$ is $\sigma$ and the restriction of $\overline{\sigma}$ on $\eta(\widetilde{\Sigma})$ is $\sigma^{-1}.$ By the definition, the map $\overline{\sigma}$ is an anti-automorphism of the quiver $\eta(\Sigma)\cup \eta(\widetilde{\Sigma})$
which has the following properties:\\
$(p_1)$ $\overline{\sigma}$ is an involution\\
$(p_2)$ $\overline{\sigma}$ maps $\mathcal{P}_{j}$ on $\mathcal{P}_{-j}$, with $0\leq j\leq k$.

     We recall that $t_{\vec{u}}$ is the translation in $\mathbb{R}^{3}$ by vector $\vec{u}=(-1,-2,0)$ defined on the step $(b)$ of the proof of the Theorem 3.1.
     We extend $\overline{\sigma}$ to $\eta(\mathbb{Z}Q)$ by the map $\tilde{\sigma}:\eta(\mathbb{Z}Q)\longrightarrow \eta(\mathbb{Z}Q)$ defined as follows:
$\tilde{\sigma}(\eta(C^{m}))=t_{{m}\vec{u}}\overline{\sigma}t_{{m}\vec{u}}(\eta(C^{m}))$ for all $m\in \mathbb{Z};$ because $\mathbb{Z}Q\cong \mathbb{Z}\widetilde{Q}$ it is equivalent to define
$\tilde{\sigma}$ by $\tilde{\sigma}(\eta(\tilde{C}^{m}))=t_{{m}\vec{u}}\overline{\sigma}t_{{m}\vec{u}}(\eta(\tilde{C}^{m}))$ for all $m\in \mathbb{Z}$. We have $t_{{m}\vec{u}}(\eta(C^{m}))=\eta(\Sigma)$ and $t_{{-m}\vec{u}}(\eta(C^{-m}))=\eta(\widetilde{\Sigma})$ for $m$ non-negative integer. We set $A_{mi}:=t_{{m}\vec{u}}(A_{ji})$, where $A_{ji}$ is defined as in the proof of the Theorem 3.1 and $\sigma_{j}\in S_{s_{j}}$ is the permutation induced by $\sigma,$ with $0\leq j\leq k.$
        To make the definition of $\tilde{\sigma}$ more specific, let $ A_{mi}$ be
a point of $\eta(C^{m})$ and  $A'_{mi}$ the image of $A_{mi}$ under the reflection $S,$ we have $\tilde{\sigma}(A_{mi})=A'_{m\sigma_{j}(i)}=S(A_{m\sigma_{j}(i)})$ if $m$ is non-negative integer, and $\tilde{\sigma}(A_{mi})=A'_{m\sigma^{-1}_{j}(i)}=S(A_{m\sigma^{-1}_{j}(i)})$ otherwise, with $0\leq j\leq k.$

           It is enough to define $\tilde{\sigma}$ on the points of $\eta(\mathbb{Z}Q)$, because the definition of $\tilde{\sigma}$  on the arrows is induced by the one on the points.
     The map $\tilde{\sigma}$ is uniquely determined by its value on the points of $\eta(\mathbb{Z}Q)$ and thus extends in the unique way to an anti-automorphism of the quiver $\eta(\mathbb{Z}Q)$.
Hence the extension $\tilde{\sigma}$ of $\sigma$ is an anti-automorphism of the quiver $\eta(\mathbb{Z}Q)$ which maps $C^{m}$ and $\tilde{C}^{-m}$ to each other.

    It remains to show that $\tilde{\sigma}$ is an involution. By the definition, it is sufficient to show this just for the points of $\eta(\mathbb{Z}Q)$.
Let $A_{mi}$ be a point of $\eta(\mathbb{Z}Q)$, if $m$ is a non-negative integer
then $$\begin{array}{lllll}\tilde{\sigma}^{2}(A_{mi})&=&\tilde{\sigma}(A'_{m\sigma_{j}(i)})\\ &=&A_{m\sigma^{-1}_{j}(\sigma_{j}(i))}\\
&=&A_{mi}; \end{array}$$
if $m$ is a negative integer then $$\begin{array}{lllll}\tilde{\sigma}^{2}(A_{mi})&=&\tilde{\sigma}(A'_{m\sigma^{-1}_{j}(i)})\\ &=&A_{m\sigma_{j}(\sigma^{-1}_{j}(i))}\\
&=&A_{mi}. \end{array}$$
The map $\tilde{\sigma}$ is an involutive anti-automorphism of $\eta(\mathbb{Z}Q).$ Thus $\tilde{\sigma}$ induces an involutive anti-automorphism of $\mathbb{Z}Q $.

\end{proof}

In order to illustrate the notations used in the proof of Lemma 4.1 we present an example.

\begin{expl}
We consider the quiver $Q$ of the Example 3.1 defined by:

$$
\xymatrix{1\ar[rr]^{\alpha_{13}}\ar[dr]_{\alpha_{12}}&&3 \\
&2\ar[ur]_{\alpha_{23}}&}.
$$
We know that $\eta(0,3)=(0,0,2)$, $\eta(0,2)=(0,1,1)$ and $\eta(0,1)=(0,1,0)$. We have $S(\eta(0,3))=(0,0,2)$, $S(\eta(0,2))=\eta(-1,2)$ and $S(\eta(0,1))=\eta(-1,1)$.
We define $\sigma$ by: the arrow  $\eta(0,\alpha_{12}):\eta(0,1)\rightarrow \eta(0,2)$ is mapped to the arrow $\eta(-1,\alpha_{12}):\eta(-1,1)\rightarrow \eta(-1,2),$ that is  $\sigma(\eta(0,1))=\eta(-1,2)$ and $\sigma(\eta(0,2))=\eta(-1,1)$; we have also
$\sigma(\eta(0,3))=\eta(0,3)$ and $\sigma(\eta(0,\alpha_{13}))=\eta(0,\alpha'_{23})$ and $\sigma(\eta(0,\alpha_{23}))=\eta(0,\alpha'_{13}).$ The map $\sigma$ is extended to an involutive anti-automorphism of $\eta(\mathbb{Z}Q).$ Hence $\sigma$ induces an involutive anti-automorphism in the transjective component  $\Gamma_{tr}\cong \mathbb{Z}Q$ of the cluster category $\mathcal{C}_{Q}$.\\
To simplify the notations, let denote by $\sigma$ the anti-automorphism of $\mathbb{Z}Q$ defined above. Let $X=\{x_1,x_2,x_3\}$ be the set of variables and $\mathcal{A}$ the cluster algebra with initial seed
$(X,Q)$. We define the map $f$ from $\mathcal{A}$ to $\mathcal{A}$ by: $f(x_1)=\frac{x_1+x_{3}+x_{2}x^{2}_{3}}{x_{1}x_{2}}$, $f(x_2)=\frac{1+x_{2}x_{3}}{x_{1}}$ and $f(x_3)=x_{3}.$ The above data defined an inverse cluster automorphism of $\mathcal{A}.$ By observing that $f(x_1)=\frac{1+f(x_{2})x_{3}}{x_{2}}$, we have $f^{2}(x_2)=\frac{1+f(x_{2})x_{3}}{f(x_{1})}=x_{1}$ and
$f^{2}(x_1)=\frac{1+f^{2}(x_{2})x_{3}}{f(x_{2})}=x_2;$ thus $f$ is an involution.\\
One can show by using the categorification of the cluster algebra $\mathcal{A}$ as in [\ref{bmrrt}] that the involutive cluster automorphism $f$ is induced by the involutive anti-automorphism $\sigma$.

\end{expl}

It has been proved in [\ref{ashifsh}, Theorem 3.11] that, if $\mathcal{A}$ is a cluster algebra with a seed
$(X,Q)$ such that $Q$ is a tree then the group of cluster automorphisms is the semidirect product Aut$\mathcal{A}$=Aut$^{+}\mathcal{A}\rtimes \mathbb{Z}_2$.
Moreover this product is not direct.
The following result shows that this holds for any finite connected acyclic quiver which is mutation equivalent
to its opposite quiver.

\begin{cor}
Let $Q$ be a finite connected acyclic quiver such that $Q$ and $Q^{op}$ are mutation equivalent. If $\mathcal{A}$ is a cluster algebra with a seed
$(X,Q)$ then the group of cluster automorphisms is the semidirect product Aut$\mathcal{A}$=Aut$^{+}\mathcal{A}\rtimes \mathbb{Z}_2$. This product is usually not direct.

\end{cor}

\begin{proof}
Let $Q$ be a finite connected acyclic quiver and  $\mathcal{A}$ a cluster algebra with a seed $(X,Q).$ We assume without loss of generality that $(X,Q)$ is the initial
seed of $\mathcal{A}$, where the quivers
$Q$ and $Q^{op}$ are mutation equivalent. We may assume that $Q$ is not Dynkin, for Dynkin type case we refer
to [\ref{ashifsh}, Lemma 3.10]. The transjective component of the cluster category is of the form $\Gamma_{tr}\cong \mathbb{Z}Q$ and
$\Gamma_{tr}\cong \mathbb{Z}Q^{op}.$ Then, there exist two local slices $\Sigma$ and $\widetilde{\Sigma}$ of $\Gamma_{tr}$ whose associated
 quivers are respectively isomorphic to  $Q$ and $Q^{op}$. We may choose $\Sigma$ and $\widetilde{\Sigma}$ such that $\Sigma_0\cap
\widetilde{\Sigma}_0\neq \emptyset$. By [\ref{ashifsh}, Lemma 3.10] the anti-automorphism $\sigma$ of the Lemma 4.1
induces an inverse cluster automorphism $f_{\sigma}$ of order two. Therefore [\ref{ashifsh}, Corollary 2.12] implies that
Aut$\mathcal{A}$=Aut$^{+}\mathcal{A}\rtimes \mathbb{Z}_2$. It was proved in [\ref{ashifsh}, Theorem 3.11] that for the particular case where
 $Q$ is a tree, this product is not direct.

\end{proof}

The existence of an involutive anti-automorphism of $\mathbb{Z}\Sigma$ give rise to a characterization of finite connected acyclic quiver
$\Sigma$ such that the translation quiver $\mathbb{Z}\Sigma$ contains a local slice whose quiver is
 isomorphic to $\Sigma^{op}.$

   Let $Q$ be a finite connected quiver and $R$ a
subquiver of $Q.$ We call the $closure$ of $R$ and we denote by $\overline{R}$ the smallest full subquiver of $Q$ containing $R.$
A $symmetric$ $quiver$ $(Q,\sigma)$ is a pair consisting of a quiver $Q=(Q_0,Q_1,s,t)$ and an involution
$\sigma:Q_{0}\amalg Q_{1}\longrightarrow
Q_{0}\amalg Q_{1}$ with $\sigma(Q_0)=Q_0$ and $\sigma(Q_1)=Q_1$ such that:\\
$(a)$ $s(\sigma(\alpha))=\sigma(t(\alpha))$ and $t(\sigma(\alpha))=\sigma(s(\alpha))$ for all $\alpha\in Q_1$\\
$(b)$ If $\sigma(t(\alpha))=s(\alpha)$ then $\sigma(\alpha)=\alpha$.

      In other words $Q$ is a symmetric quiver if it has an involutive anti-automorphism. For more details about symmetric quivers, we refer to [\ref{deway}].
\begin{definition}
Let $Q$ be a connected quiver and $(Q^{i})_{1\leq i\leq m}$ a family of full subquivers of $Q$. We say that the family $(Q^{i})_{1\leq i\leq m}$
is a complete decomposition
of $Q$ if the
following conditions are satisfied:\\
(C1) $Q_0^{i}\cap Q_0^{j}=\phi$ if $i\neq j$\\
(C2) $Q=\overline{Q^{1}\cup Q^{2}...Q^{n}}.$\\
(C3) The only arrows between two connected components are from $Q^{i}$ to $Q^{i-1}$ for $1\leq i\leq n-1$\\
\end{definition}

\begin{cor}
Let $Q$ be a finite connected acyclic quiver such that the translation quiver $\mathbb{Z}Q$ contains a local slice whose quiver is
 isomorphic to $Q^{op}.$ Then $Q$ possesses a complete decomposition into a family of symmetric subquivers.
\end{cor}

  \begin{proof}
Let $Q$ be a finite connected acyclic quiver, and $\Sigma$ be a local slice in $\mathbb{Z}Q$ whose quiver is isomorphic to $Q$ and $\widetilde{\Sigma}$
be a local slice in $\mathbb{Z}Q$ whose quiver is isomorphic to $Q^{op}$.
   We consider the spatial representation of $\mathbb{Z}Q$.
Because of Theorem 3.1 and the
definition of the spatial representation of $\mathbb{Z}Q$, we assume that $\eta(\Sigma)$ intersects $k+1$
planes $\mathcal{P}_{m}:y=m,$ where $0\leq m\leq k.$
We have $\eta(\Sigma^{l})=\eta(\Sigma)\cap \mathcal{P}_{l}$, where $\Sigma^{l}$ is the full subquiver of $\eta(\Sigma)$ represented in
 the plane
$\mathcal{P}_{l}:y=l,$ via the embedding $\eta$ of Theorem 3.1. By construction, the family $(\Sigma^{l})_{0\leq l\leq k}$ is a complete decomposition
of $\Sigma.$ It remains to prove that $\Sigma^{l}$ is a symmetric quiver for
$0\leq l\leq k.$
We have $\eta(\widetilde{\Sigma}^{l})=\eta(\widetilde{\Sigma})\cap \mathcal{P}_{-l}$, where $\eta(\widetilde{\Sigma}^{l})$ is the full subquiver of $\eta(\widetilde{\Sigma})$ represented in
 the plane
$\mathcal{P}_{l}:y=-l,$ via the embedding $\eta$ of Theorem 3.1. So the family $(\widetilde{\Sigma}^{l})_{0\leq l\leq k}$ is a complete decomposition
of $\widetilde{\Sigma}.$

        Since $\widetilde{\Sigma}$ is a local slice of $\mathbb{Z}Q$, then by the
Lemma 4.1, there exists an involutive
anti-isomorphism $\sigma:\mathbb{Z}Q\longrightarrow \mathbb{Z}Q$ such that $\sigma(\Sigma)=\widetilde{\Sigma}.$

Because $S(\eta(\Sigma^{l}))$ and $\eta(\widetilde{\Sigma}^{l})$ are two full subquivers of $\eta(\widetilde{\Sigma})$ located in the same plane with $|S(\eta(\Sigma^{l}))|=|\eta(\widetilde{\Sigma}^{l})|,$
we deduce that the coincide: $S(\eta(\Sigma^{l}))=\eta(\widetilde{\Sigma}^{l}).$ Since $S(\eta(\Sigma^{l}))\cong \eta(\Sigma^{l})$ and $\eta(\widetilde{\Sigma}^{l})\cong (\eta(\Sigma^{l}))^{op}$, we get that
$\eta(\Sigma^{l})\cong (\eta(\Sigma^{l}))^{op}.$ Thus the restriction of $\sigma$ on $\eta(\Sigma^{l})$ is an involutive anti-automorphism.
 Hence $\Sigma^{l}$ is a symmetric quiver for all $l$, with
$0\leq l\leq k.$ This completes the proof of our assertion.
\end{proof}

\begin{cor}
Let $Q$ be a finite connected acyclic quiver and $K$ an algebraically closed field. Then
$Q$ and $Q^{op}$ are mutation equivalent if and only if $KQ$ and $KQ^{op}$ are derived equivalent.

\end{cor}

\begin{proof}
Let $Q$ be a finite connected acyclic quiver and $K$ an algebraically closed field, $KQ$ the path algebra.
Suppose that $Q$ and $Q^{op}$ are mutation equivalent. Then  Corollary 4.1 implies that there exists an inverse automorphism of the cluster algebra
$\mathcal{A}$ denoted by $\sigma.$ Then, because of [\ref{ashifsh}, Corollary 3.1]
$\sigma$ induces a triangle equivalence $\sigma_{\mathcal{D}}:\mathcal{D}^{b}$(mod$KQ$)$\longrightarrow \mathcal{D}^{b}$(mod$KQ^{op}$).
Conversely, suppose that $KQ$ and $KQ^{op}$ are derived equivalent, the cluster category $\mathcal{C}_{Q}$ of the finite connected acyclic quiver $Q$ is an orbit category of
$\mathcal{D}^{b}$(mod$KQ$). It was proved by Keller in [\ref{kel}, Theorem 1] that the canonical projection $\pi :\mathcal{D}^{b}$(mod$KQ)\longrightarrow \mathcal{C}_{Q}$ is a triangle functor.
Therefore $\mathcal{C}_{Q}$  and
$\mathcal{C^{op}}_{Q}$ are triangulated equivalent. Thus, according to [\ref{kelrei}], $Q$ and $Q^{op}$ are mutation equivalent.

\end{proof}

      \addcontentsline{toc}{chapter}{Bibliographie}

$$   $$

D\'{e}partement de Mathématiques, Universit\'{e} de Sherbrooke
2500,boul. de l'Universit\'{e},
Sherbrooke, Qu\'{e}bec, J1K2R1
Canada
.\\
E-mail address: ndoune.ndoune@usherbrooke.ca

\end{document}